\documentclass[11pt]{amsart}

\usepackage{amsmath,amsthm,amsfonts,amssymb,bm,graphicx }

\usepackage{graphicx}
\usepackage{enumitem}
\usepackage{xparse}
\usepackage{tikz}
\usepackage{comment}
\usepackage{bm}

\usepackage
{hyperref}
\hypersetup{colorlinks=true,citecolor=blue,linkcolor=blue,urlcolor=blue}

\makeatletter
\g@addto@macro\bfseries{\boldmath}
\makeatother

\theoremstyle{plain}
\newtheorem{theorem}{Theorem}
\newtheorem*{theorem*}{Theorem}
\newtheorem{lemma}[theorem]{Lemma}

\theoremstyle{remark}

\numberwithin{theorem}{section}

\numberwithin{equation}{section}

\def\N{\mathbb N}
\def\Z{\mathbb Z}
\def\R{\mathbb R}
\def\Q{\mathbb Q}
\def\C{\mathbb C}

\def\O{\mathcal O}
\def\P{\mathcal P}
\def\bt{\blacktriangle}
\def\ve{\varepsilon}

\begin{document}


\author{Magdal\'ena Tinkov\'a}

\title{On the Pythagoras number of the simplest cubic fields}

\address{Charles University, Faculty of Mathematics and Physics, Department of Algebra,
Sokolovsk\'{a} 83, 18600 Praha 8, Czech Republic}

\email{tinkova.magdalena@gmail.com}


\thanks{The author was supported by Czech Science Foundation GA\v{C}R, grant 21-00420M, by projects PRIMUS/20/SCI/002, UNCE/SCI/022, GA UK 1298218 from Charles University, and by SVV-2020-260589.}


\begin{abstract}
The simplest cubic fields $\mathbb{Q}(\rho)$ are generated by a root $\rho$ of the polynomial $x^3-ax^2-(a+3)x-1$ where $a\geq -1$. In this paper, we will show that the Pythagoras number of the order $\mathbb{Z}[\rho]$ is equal to $6$ for $a\geq 3$.

\end{abstract}

\setcounter{tocdepth}{2}  \maketitle 

\section{Introduction}

Let $\O$ be a commutative ring, and let $\sum \O^2$ and $\sum^m\O^2$ be the sets defined by
\[
\sum\O^2=\Big\{\sum_{i=1}^n \alpha_i^2;\; \alpha_i\in\O, n\in\N\Big\},\qquad \sum^m\O^2=\Big\{\sum_{i=1}^m \alpha_i^2;\; \alpha_i\in\O\Big\}.
\]
In this paper, we are concerned with the so-called Pythagoras number $\P(\O)$ of the ring $\O$ given by
\[
\P(\O)=\min\Big\{m\in\N;\; \sum \O^2=\sum^m \O^2\Big\}.
\]
Regarding some basic examples, $\P(\R)=\P(\C)=1$, and Lagrange's famous four-square theorem implies $\P(\Q)=4$. Moreover, it can be proved that $\P(K)\leq 4$ for every number field $K$ \cite{Ho, Si1}. 

Arguably the most important and classical cases are Pythagoras numbers of rings of algebraic integers $\O_K$ of totally real number fields $K$. The first result is, of course, Lagrange's above-mentioned theorem giving $\P(\Z)=4$ that led to the study of universal quadratic forms. Let $\O_K^{+}$ be the set of totally positive integers of $\O_K$ (by this, we mean those algebraic integers whose conjugates are all positive). Roughly speaking, universal quadratic form over $\O_K$ is a quadratic form which has coefficients from $\O_K$ and which represents all the elements in $\O_K^{+}$. For more details about universal quadratic forms, see also for example \cite{BK1,BK2,CKR,HKK,Ka,Ki1,Ki2,Sa}.  

Considering sums of squares, Maa{\ss} has shown that the sum of three squares is universal over $\O_K$ for $K=\Q(\sqrt{5})$, which implies $\P(\O_K)=3$ in this case \cite{Ma}. Nevertheless, the following result of Siegel says that a sum of any number of squares can be universal only in the fields $\Q$ and $\Q(\sqrt{5})$ \cite{Si2}. It means that in the other totally real number fields, we cannot express all the elements of $\O_K^{+}$ as a sum of squares, and thus we must restrict to those which indeed lie in $\sum\O_K^2$.

Let now $\O\subseteq \O_K$ be an order. Scharlau showed that the Pythagoras number of an order is always finite, although it can be arbitrarily large \cite{Sc}. The case of quadratic orders was in great detail studied by Peters; he proved that except for a few cases, the Pythagoras number is always $5$ \cite{Pet}. Moreover, he also characterized all the elements which are representable as a sum of squares.   
Considering the other cases, recently, Kala and Yatsyna \cite{KY2} proved that $\P(\O)\leq f(d)$ for every order $\O$ in totally real number $K$, where $f$ is a function depending only on the degree $d$ of the field $K$. Moreover, one can take $f(d)=d+3$ if $2\leq d\leq 5$. Note that their subsequent paper \cite{KY1} studies sums of squares in certain subrings of $\O_K$.      

However, given the difficulty of studying $\P(\O)$ for orders, most of the research so far focused on the situation over fields.
In the case of non-formally real fields $K$ (i.e., in which $-1$ can be expressed as a sum of squares), the Pythagoras number is closely related to Stufe $s(K)$ of $K$, which is the minimal number of squares whose sum gives $-1$. We have here $s(K)\leq \P(K)\leq s(K)+1$. By the results of Pfister \cite{Pf}, the value of $s(K)$ can attain only the powers of $2$, which greatly limits the possibilities for the value of $\P(K)$. On the other hand, Hoffmann has shown that for every $n\in\N$ and formally real field $K_0$, there exists a formally real field $K$ over $K_0$ with $\P(K)=n$ \cite{Ho}.
Nevertheless, we can find many other results on the Pythagoras number of fields in the number of specific situations, for example, in relation with rational function fields, elliptic curves, Hasse number or Laurent series \cite{CEP, CDL, Hu, Pre}.

In this paper, we will focus on orders in the so-called simplest cubic fields \cite{Co,Sh}. They are generated by a root $\rho$ of the polynomial $x^3-ax^2-(a+3)x-1$ where $a\geq -1$, and were richly studied in many different contexts, see for example \cite{Ba, By, Fo, LP, Let, Lo, Wa}. This is due to the fact that they have many useful properties: They contain units of all signatures, and every totally positive unit is a square \cite{N}. Moreover, $\O_K=\Z[\rho]$ for infinitely many $a$ (for example, if the square root $a^2+3a+9$ of the discriminant is squarefree), and they are also cyclic.

In this case, the result of Kala and Yatsyna gives the upper bound $6$ on $\P(\O)$. We will show that this bound is attained in infinitely many cases by proving the following theorem:      

\begin{theorem} \label{thm:main}
Let $\rho$ be a root of the polynomial $x^3-ax^2-(a+3)x-1$ where $a\geq 3$. Then $\mathcal{P}(\Z[\rho])=6$.
\end{theorem}

To the best of our knowledge, there are no results on the Pythagoras number for orders in number fields of higher degrees similar to Peters' results on quadratic orders. Thus, Theorem \ref{thm:main} represents the first breakthrough in this problem. Moreover, since $\O_K=\Z[\rho]$ in infinitely many cases of $a$, this conclusion also provides us a precise result for the maximal order $\O_K$.

Having the upper bound from the result of Kala and Yatsyna, we will focus on the determination of the lower bound. To reach this aim, we will primarily rely on the idea of additively indecomposable integers in totally real algebraic fields. Probably the most studied case is when we consider totally positive elements. Let $\alpha\in\O_K^{+}$. We say that $\alpha$ is indecomposable in $\O_K$ if we cannot express it as $\alpha=\beta_1+\beta_2$ where $\beta_1,\beta_2\in\O_K^{+}$. Note that under the name \textit{extremal elements}, they can be found in the above-mentioned Siegel's proof of the (non)-universality of sums of squares in number fields. However, in our proofs, we will need their extended definition for all the possible signatures (see Section \ref{Sec:Preli}). 

Regarding real quadratic fields, the indecomposable integers were fully described by Perron \cite{Pe} and Dress and Scharlau \cite{DS}, and their additive structure was studied in \cite{HK}. Some partial results for the biquadratic case can be found in the work of \v{C}ech, Lachman, Svoboda, Zemková and the present author \cite{CLSTZ}, and in the following paper \cite{KTZ}, which focuses on ternary quadratic forms in these fields. The cubic fields are in the center of interest of \cite{KT}, where we also determined the full structure of indecomposable integers in the simplest cubic fields. The proof of Theorem \ref{thm:main} is based on this result.    
     
Nevertheless, so far, the indecomposable integers have been mainly used in the study of universal quadratic forms \cite{ BK1, BK2, CLSTZ, Ka, KT, KTZ, Si2, Ya} or the elements of small norms \cite{KT}, thus this paper also provides a new application of this phenomenon. Moreover, some of the ideas introduced here can be also used for the determination of the Pythagoras number for other cubic orders. 

The proof of Theorem \ref{thm:main} is provided in Section \ref{Sec:proofmain}. Moreover, in Section \ref{Sec:asmall} we give some partial results on the Pythagoras number for the remaining cases $-1\leq a\leq 2$.      

\section{Preliminaries} \label{Sec:Preli}

Let $K=\Q(\rho)$ be a totally real cubic field, and let $\O_K$ be the ring of algebraic integers of $K$. Moreover, let $\rho'$ and $\rho''$ be Galois conjugates of $\rho$. Then by signature $\sigma$ of $\alpha\in\Q(\rho)$, we mean the triple 
\[
(\text{sgn}(\alpha),\,\text{sgn}(\alpha'),\,\text{sgn}(\alpha''))
\]
where $\text{sgn}$ is the signum function, and $\alpha'$ and $\alpha''$ are images of $\alpha$ under the $\Q$-isomorphism given by $\rho\mapsto\rho'$, and $\rho\mapsto\rho''$, respectively. 
In the following, we will use symbols $+$ and $-$ instead of $\pm1$, e.g., we will replace $(1,1,1)$ by $(+,+,+)$. Moreover, the norm of $\alpha$ is defined as $N(\alpha)=\alpha\alpha'\alpha''$, and the trace of $\alpha$ as $\text{Tr}(\alpha)=\alpha+\alpha'+\alpha''$. 

Let $\O\subseteq\O_K$ be an order in $K$, i.e., a subring of finite index in $\O_K$. By $\O^{\sigma}$, we will mean the set of those elements in $\O$ which have the signature $\sigma$. The element $\alpha\in \O^{\sigma}$ is $\sigma$-indecomposable in $\O$ if it cannot be written as $\alpha=\beta_1+\beta_2$ where $\beta_1,\beta_2\in\O^{\sigma}$. Otherwise, we say that the element $\alpha$ is $\sigma$-decomposable in $\O$. Note that for example, all the units are $\sigma$-indecomposable for some signature $\sigma$.

In particular, the totally positive elements, i.e., the elements with the signature $(+,+,+)$, were richly studied in the past, and for them, we will introduce some more notation. We will denote the subset of totally positive elements of $\O$ by $\O^{+}$. We say that $\alpha\in\O^+$ is totally greater than $\beta\in\O^+$ if $\alpha>\beta$, $\alpha'>\beta'$ and $\alpha''>\beta''$. We will denote it by $\alpha\succ\beta$. Sometimes, we will also use the symbol $\succeq$ when we want to include the case when $\alpha=\beta$. Note that, for example, all non-zero squares are totally positive. 

Let us now recall some facts about the simplest cubic fields, which we study in this paper. They are generated by a root of the polynomial $x^3-ax^2-(a+3)x-1$.
Troughout this paper, we will denote the roots of this polynomial in the following way: $a+1<\rho$, $-2<\rho'<-1$, and $-1<\rho''<0$. Nevertheless, if $a\geq 7$, we have more precise estimates on these roots, namely  
\begin{equation} \label{eq:estvalues}
a+1<\rho<a+1+\frac{2}{a},\ \ -1-\frac{1}{a+1}<\rho'<-1-\frac{1}{a+2},\text{ and } -\frac{1}{a+2}<\rho''<-\frac{1}{a+3}. 
\end{equation}
Note that this result mostly comes from \cite{LP}, only the original estimate for $\rho'$ was too rough for the purposes of this paper, so we have stated its slightly improved form, which can be easily checked. We will use these estimates many times in the following proofs.

Besides that, we will use the fact that we know the full structure of $\sigma$-indecomposable integers in $\Z[\rho]$. In particular, in \cite{KT}, we have shown the following theorem:     

\begin{theorem}[{\cite[Theorem 1.2]{KT}}] \label{thm:indesimplest}
Let $K$ be the simplest cubic field with
$a\in\Z_{\geq -1}$ such that $\O_K=\Z[\rho]$. 
The elements $1$, $1+\rho+\rho^2$, and $-v-w\rho+(v+1)\rho^2$ where $0\leq v\leq a$ and $v(a+2)+1\leq w\leq (v+1)(a+1)$ are, up to multiplication by totally positive units, all the totally positive indecomposable integers in $\Q(\rho)$.
\end{theorem}

Note that in fact, Theorem \ref{thm:indesimplest} provides us all the totally positive indecomposable integers in the order $\Z[\rho]$ for every $a\geq -1$. Moreover, although this theorem considers only the totally positive indecomposable integers, it gives us also the complete information about $\sigma$-indecomposables for all the other signatures $\sigma$. These $\sigma$-indecomposables can be obtained as $\varepsilon\alpha$ where $\varepsilon$ runs over all the units with signature $\sigma$, and $\alpha$ runs over all the elements listed in Theorem \ref{thm:indesimplest}. This property is given by the fact that $\Q(\rho)$ contains units of all signatures. 

Moreover, we can divide the totally positive indecomposable integers from Theorem \ref{thm:indesimplest} into three sets: units, the exceptional indecomposable integer $1+\rho+\rho^2$ and the ``triangle" of indecomposables of the form
\[
\bt=\bt(a)=\{-v-w\rho+(v+1)\rho^2, 0\leq v\leq a \text{ and } v(a+2)+1\leq w\leq (v+1)(a+1)\}.
\]
Nevertheless, except for $\alpha\in\bt$, the set $\bt$ also contains some specific unit multiples of conjugates of $\alpha$. Let $\alpha(v,W)=-v-(v(a+2)+1+W)\rho+(v+1)\rho^2\in\bt$ for some $0\leq W\leq a-v$, and let $a=3A+a_0$ where $a_0\in\{0,1,2\}$. Instead of $\bt$, we can consider its subset of the form
\[
\bt_0=\bt_0(a)=\left\{
\begin{array}{ll}
\left\{\alpha(v,W);0\leq v\leq A\text{ and } v\leq W\leq a-2v-1\right \}\text{ if } a_0\in\{1,2\},\\
\{\alpha(v,W);0\leq v\leq A-1\text{ and } v\leq W\leq a-2v-1\}\cup\{\alpha(A,A)\}\text{ if } a_0=0.
\end{array}
\right.
\]
The excluded elements of $\bt$ are just these unit multiples of conjugates of $\bt_0$, and thus in some sense, covered by the elements in $\bt_0$. For more details, see \cite{KT}.

In our proof, we will work with norms of these elements, and in particular, we will use the following lemma from \cite{KT}, which partly compares norms of elements belonging to the set $\bt_0$.   

\begin{lemma}[{\cite[Lemma 6.4]{KT}}] \label{lemma:norm_ineq}
	Let $a\geq 3$ and assume that $\alpha(v+1,W)\in\bt_0$. Then $N(\alpha(v,W))<N(\alpha(v+1,W))$.
\end{lemma}

Note that for fixed $v$, the norm of $\alpha(v,W)$ firstly increases in $W$ and then it can start to decrease (in some cases, it increases in the whole interval for $W$ but one of these two cases always occurs). For more details, see \cite{KT}.

As we will see below, we will also need to know more about units in $\Q(\rho)$. It was proved that the system of fundamental units of $\Q(\rho)$ (and also of $\Z[\rho]$) is formed by the pair $\rho$ and $\rho'$ \cite{God,Sh}. Benefiting from this property, the authors of \cite{KT} prove the following lemma, which we will often use in this paper.  

\begin{lemma}[{\cite[Lemma 6.2]{KT}}] \label{lemma:units>a^2}
	Let $a\geq 7$ and let $\varepsilon$ be a unit such that $|\varepsilon|,|\varepsilon'|,|\varepsilon''|<a$. Then $\varepsilon=1$.
\end{lemma}

Especially, if $\varepsilon\neq 1$ is a totally positive unit (and thus a square), then by Lemma \ref{lemma:units>a^2}, some of its conjugates is greater than $a^2$. In some cases, we will also need the stronger result stated in the following lemma \cite{KT}. 

\begin{lemma} [{\cite[Lemma 6.3]{KT}}] \label{lemma:a^4/a^2+a^2}
	Let $a\geq 7$ and let $\ve$ be a totally positive unit such that $\ve>a^2$. If $\ve\neq \rho^2,\rho''^{-2}$, then at least one of the following holds:
	\begin{enumerate}
		\item $\ve>a^4$, or
		\item $\ve'>a^2$, or
		\item $\ve''>a^2$.
	\end{enumerate} 
\end{lemma}

\section{Proof of Theorem \ref{thm:main}} \label{Sec:proofmain}

Now we will describe the method which we will use in the proof of Theorem \ref{thm:main}. Recall that by the result of Kala and Yatsyna \cite{KY1}, the upper bound on the Pythagoras number in cubic orders is $6$. Thus, it suffices to prove that the lower bound is also $6$. To do that, it is enough to find an element $\gamma\in\Z[\rho]$ which can be written as a sum of six squares but not as a sum of five squares.

Hence we will proceed as follows. We will suitably choose such an element $\gamma$ and find all the elements $\omega$ such that $\gamma\succeq \omega^2$. Every square decomposition of $\gamma$ can consist only of these elements. Then, using some combinatorics, we will show that none sum of five (or less) of these squares can give $\gamma$.

In the determination of these squares, we will use the knowledge of $\sigma$-indecomposable integers in the simplest cubic fields originating from Theorem \ref{thm:indesimplest}. Let $\omega$ be such that $\gamma\succeq \omega^2$. This element $\omega$ has some signature $\sigma$, and it can be thus expressed as $\omega=\sum_{i=1}^n\beta_i$ where $\beta_i$ are $\sigma$-indecomposable integers in $\Z[\rho]$, and $n\in\N$. Having this, we can see that
\[
\gamma\succeq \sum_{i=1}^n\beta_i^2+2\sum_{\substack{i,j=1\\i\neq j}}^n\beta_i\beta_j.
\]
Obviously, the squares $\beta_i^2$ are totally positive, as well as elements $\beta_i\beta_j$ for $i\neq j$ since $\beta_i$ and $\beta_j$ have the same signature $\sigma$. Thus, we can immediately conclude that $\gamma\succeq \beta_i^2$ for all $i=1,\ldots,n$. We will use this simple fact in the following way. First of all, we will find all the $\sigma$-indecomposable integers $\beta$ for all the signatures $\sigma$ such that $\gamma\succeq \beta^2$. Then, by summing these elements $\beta$ with the same signature $\sigma$, we will derive all the $\sigma$-decomposable integers $\omega$ satisfying $\gamma\succeq \omega^2$.

Moreover, every of these $\sigma$-indecomposables $\beta$ can we rewritten as $\beta=\varepsilon\alpha$ where $\varepsilon$ is a unit and $\alpha$ is one of $1$, $1+\rho+\rho^2$ and elements of $\bt_0$, or one of their conjugates. Thus, we firstly detect the elements of this list whose squares have the norm smaller than $N(\gamma)$, and consequently use the results on units from Lemmas \ref{lemma:units>a^2} and \ref{lemma:a^4/a^2+a^2} to determine all the possible units $\varepsilon$ which indeed give $\beta^2=\varepsilon^2\alpha^2\preceq \gamma$.

In the case of the simplest cubic fields, we can choose our element as
\[
\gamma=a^2+a+8+(a^2-a+1)\rho+(2-a)\rho^2=1+1+1+4+\rho^2+(a+1+a\rho-\rho^2)^2.
\]
As we see, we can write $\gamma$ as a sum of six squares. We will fix this choice of $\gamma$ and work with it for the rest of this paper. Moreover, we will show that except for a few cases of $a$, there exist only $8$ non-zero elements $\omega^2$ such that $\omega^2\preceq \gamma$, which is a great advantage of the choice of this element.  

Using estimates given in (\ref{eq:estvalues}), we can easily deduce that for $a\geq 7$,
\begin{align*}
\gamma&<a^2+6a+9+\frac{2}{a},\\
\gamma'&<10+\frac{4}{(a+2)^2}<11,\\
\gamma''&<a^2+11+\frac{a^2-8a-28}{(a+3)^2}.
\end{align*}
In particular, the conjugate $\gamma$ has the largest value.
We can immediately see that, if $\omega$ is a rational integer, then necessarily $\omega^2\in\{1,4,9\}$. 

\subsection{Units} \label{subsec:units}

Our first concern is to find all the totally positive units $\varepsilon$ satisfying $\gamma\succeq\varepsilon$. Recall that every such unit is a square, thus it can play a role in a square decomposition of the element $\gamma$.

\begin{lemma} \label{lemma:units}
Let $a\geq 7$ and let $\varepsilon$ be a totally positive unit in $\Z[\rho]$. If $\gamma\succeq \varepsilon$, then $\varepsilon\in\{1,\rho^2\}$. 
\end{lemma}

\begin{proof}
If $\varepsilon\neq 1$, Lemma \ref{lemma:units>a^2} implies that one of the conjugates of $\varepsilon$ is greater than $a^2$. Without loss of generality, we can assume $\varepsilon>a^2$. Using the fundamental units, the unit $\varepsilon$ can be written as $\rho^k\rho'^l$ for some $k,l\in 2\Z$. As we can see from the estimates in (\ref{eq:estvalues}), the value of $\varepsilon$ can be greater than $a^2$ only if $k\geq 2$. On the other hand, $\varepsilon'=\rho'^k\rho''^l\leq \gamma$ for $k\geq 2$ only if $l\geq -2$. This condition on $l$ also implies that $\varepsilon\leq \gamma$ only for $k=2$ and $a\geq 7$. The other cases are not possible as we have $\gamma>\gamma',\gamma''$.    

Let us first focus on the case when $k=2$ and $l=-2$. Obviously, $\varepsilon'>\gamma''$, thus the only conjugate of $\varepsilon$ which can be totally smaller $\gamma$ is $\varepsilon'$. However, it can be directly verified that $\gamma-\rho'^2\rho''^{-2}$ is not totally positive for $a\geq 7$.  

Therefore, let $l\geq 0$. In these cases, clearly, $\varepsilon>\gamma''$ for $a\geq7$, thus $\varepsilon$ is the only conjugate of $\varepsilon$ which can be totally smaller than $\gamma$. First of all, let us assume $l\geq 6$. In this case, we have
\[
\varepsilon''=\rho''^2\rho^l>\frac{1}{(a+3)^2}(a+1)^6>\gamma''
\]
for $a\geq 7$, and thus we can exclude the cases with $l\geq 6$. 

Therefore, except for $\varepsilon=1$, we are left with the units $\rho^2$, $\rho^2\rho'^2$ and $\rho^2\rho'^4$. However, the last unit can be rewritten as $\rho'^2\rho''^{-2}$, which was excluded in the previous part. In the same manner, we can show that $\gamma-\rho^2\rho'^2$ is not totally positive for $a\geq 7$. Thus, the only units which can be (and actually are) totally smaller than $\gamma$ are exactly $1$ and $\rho^2$.   
\end{proof}

\subsection{Squares of $\sigma$-indecomposable integers} \label{subsec:inde}

In this part, we will find all non-unit $\sigma$-indecomposable integers $\beta$ such that $\gamma\succeq \beta^2$. Necessarily, in that case, $N(\gamma)\geq N(\beta^2)$. It can be easily computed that 
\[
N(\gamma)=9a^4+22a^3+247a^2+258a+1493.
\]

In our investigation, we can use the knowledge of totally positive indecomposable integers given by Theorem \ref{thm:indesimplest}, and the fact that every square of non-unit $\sigma$-indecomposable integer is a conjugate of some element of the form $\varepsilon\alpha^2$ where $\varepsilon$ is a totally positive unit and $\alpha\in\bt_0\cup\{1+\rho+\rho^2\}$. Thus, we firstly detect all the elements $\alpha$ for which $N(\alpha^2)=N(\beta^2)\leq N(\gamma)$. 

\begin{lemma} \label{lemma:smallnorms}
Let $a\geq 15$ and let $\alpha\in\bt_0\cup\{1+\rho+\rho^2\}$. If $ N(\alpha^2)\leq N(\gamma)$, then
\[
\alpha\in\{-w\rho+\rho^2;1\leq w\leq a\}\cup\{1+\rho+\rho^2,-1-(a+4)\rho+2\rho^2\}.
\]
\end{lemma}

\begin{proof}
It can be easily computed that $N((1+\rho+\rho^2)^2)<N(\gamma)$ for $a\geq -1$. Thus, let us now focus on $\alpha\in\bt_0$. In this case, we have $\alpha=\alpha(v,W)=-v-(a(v+2)+1+W)\rho+(v+1)\rho^2$ for some admissible values of $v,W$. In the following, we will use Lemma \ref{lemma:norm_ineq}, which compares norms of elements belonging to $\bt_0$. 

Firstly, let us focus on the case when $v=1$. For $W=1$ (the smallest value of $W$ for $v=1$), we get $N(\alpha(1,1)^2)=4a^4+24a^3-108a+81<N(\gamma)$ for $a\geq -1$, i.e., we obtain the element $-1-(a+4)\rho+2\rho^2$ listed in the statement of the lemma. On the other hand, $N(\alpha(1,2)^2)=9a^4+54a^3-141a^2-666a+1369>N(\gamma)$ for $a\geq 15$. Recall that the norm of $\alpha(v,W)$ for fixed $v$ increases in $W$, and then it can start to decrease. Thus, to complete the proof for $v=1$, it suffices to check the norm for $W=a-3$ (the largest $W$ for $v=1$ and $a\geq 15$). Nevertheless, we obtain $N(\alpha(1,a-3)^2)=16a^4-136a^2+289>N(\gamma)$ for $a\geq 10$.

By Lemma \ref{lemma:norm_ineq} and using the previous part, the norms of $\alpha(v,W)^2$ for $v\geq 2$ are too large to be smaller than $N(\gamma)$. Thus, we are left with one element with $v=1$ and all the elements with $v=0$, which completes the proof.          
\end{proof}

Therefore, we have determined all the representatives of the $\sigma$-indecomposable integers $\alpha$ with sufficiently small norms. Now we will find all the totally positive units $\varepsilon$ for which $\varepsilon\alpha^2$ or one of its conjugates is indeed totally smaller than $\gamma$. To reach this aim, we will use Lemmas \ref{lemma:units>a^2} and \ref{lemma:a^4/a^2+a^2}, which state some useful results about units in the simplest cubic fields.

\begin{lemma} \label{lemma:inde}
Let $a\geq 15$ and let $\beta$ be a non-unit $\sigma$-indecomposable integer in $\Z[\rho]$ for some signature $\sigma$. If $\gamma\succeq\beta^2$, then $\beta^2$ is one of the following elements: 
\begin{enumerate}
\item $\rho'^2\rho''^2(-\rho+\rho^2)^2=1-2\rho+\rho^2$,
\item $\rho''^2\rho^2(-\rho'+\rho'^2)^2=a^2+a+1+(a^2-a+1)\rho-(a-1)\rho^2$,
\item $\rho''^2\rho^2(-2\rho'+\rho'^2)^2=a^2-a+(a^2-3a+1)\rho-(a-3)\rho^2$,
\item $\rho^2\rho'^2(-(a-1)\rho''+\rho''^2)^2=a^2+a-1+(a^2-a-3)\rho-(a-2)\rho^2$.
\end{enumerate}
\end{lemma}

\begin{proof}
We will proceed as follows. We will consider the elements $\alpha$ given by Lemma \ref{lemma:smallnorms} and discuss whether some conjugate of $\varepsilon\alpha^2$ can be totally smaller than $\gamma$ for some totally positive unit $\varepsilon$.
 
Let us start with $\alpha=1+\rho+\rho^2$. Using (\ref{eq:estvalues}), we can see that
\[
\alpha^2>a^4+6a^3+15a^2+18a+9>a^2+6a+9+\frac{2}{a}>\gamma>\gamma',\gamma''
\]
for $a\geq 15$.
Thus, we can immediately exclude all the conjugates of $(1+\rho+\rho^2)^2$, i.e., when we multiple $\alpha^2$ by the totally positive unit $\varepsilon=1$. Let now $\varepsilon\neq 1$. 
By Lemma \ref{lemma:units>a^2}, without loss of generality, we can suppose $\varepsilon>a^2$. However, we can immediately exclude the conjugates of $\varepsilon(1+\rho+\rho^2)^2$, and we are left with the elements of the form $\varepsilon(1+\rho'+\rho'^2)^2$ and $\varepsilon(1+\rho''+\rho''^2)^2$. 

Now, we will use Lemma \ref{lemma:a^4/a^2+a^2}. If $\varepsilon>a^4$, it can be easily verified that $\varepsilon(1+\rho'+\rho'^2)^2,\varepsilon(1+\rho''+\rho''^2)^2>\gamma$, thus these cases are not possible. Suppose $\varepsilon\neq \rho^2,\rho''^{-2}$. In that case, our unit $\varepsilon$ has a conjugate greater than $a^2$. Since this conjugate cannot be paired with $1+\rho+\rho^2$, we can restrict to the elements $\varepsilon(1+\rho'+\rho'^2)^2$ with $a^2<\varepsilon,\varepsilon'<a^4$ (the case $\varepsilon(1+\rho''+\rho''^2)^2$ is covered by that). However, using a similar method as in Lemma \ref{lemma:units}, we can show that under these conditions, $\varepsilon=\rho^2\rho'^{-2}$. Nevertheless, it can be easily checked that none conjugate of $\rho^2\rho'^{-2}(1+\rho'+\rho'^2)^2$ is totally smaller than $\gamma$. 

Thus we are left with the elements of the form $\varepsilon(1+\rho'+\rho'^2)^2$ and  $\varepsilon(1+\rho''+\rho''^2)^2$ where $\varepsilon = \rho^2,\rho''^{-2}$. However, for these cases, we can directly verify that none of them (or their conjugates) is totally smaller than $\gamma$.  

\bigskip

We will proceed with the elements from $\bt_0$. First of all, let us assume that $\alpha=-w\rho+\rho^2$ where $3\leq w\leq a-2$. In this case, as before, $\alpha^2>\gamma$ (as $\alpha^2>((a+1)^2-w(a+2))^2\geq 4a^2+20a+25>\gamma$), thus we can exclude $\varepsilon=1$. Let now $\varepsilon>a^2$. 
Obviously, $\varepsilon\alpha^2>\gamma$ , and $\varepsilon\alpha'^2>\gamma$ as $\varepsilon\alpha'^2>a^2(w+1)^2\geq 16a^2>\gamma$. Hence only the conjugates of $\varepsilon\alpha''^2$ can be totally smaller than $\gamma$. Similarly as before, using Lemma \ref{lemma:a^4/a^2+a^2}, we can exclude all the units with $\varepsilon>a^4$, $\varepsilon'>a^2$ and $\varepsilon''>a^2$, and neither of $\rho^2$ or $\rho''^{-2}=\rho^2\rho'^2$ produce an element totally smaller than $\gamma$, which can checked by direct computations.   

Put $\alpha=-1-(a+4)\rho+2\rho^2$. Using Lemma \ref{lemma:units>a^2} and basic estimates (\ref{eq:estvalues}), we can easily find candidates on totally smaller integers, which are conjugates of $\varepsilon\alpha''^2$ where $\varepsilon\in\{\rho^2\rho'^2,\rho^4\rho'^2,\rho^4\rho'^4\}$. Nevertheless, by comparing norms, traces, or the remaining coefficient of the minimal polynomial of $\gamma$ and $\alpha$, we can exclude all of these several concrete cases. 

Let now $\alpha=-a\rho+\rho^2$. In this case, $\alpha>a^2-2a+1$, $\alpha'>(a+1)^2$ and $\alpha''>\frac{a^2}{(a+3)^2}$. Nevertheless, $\text{Tr}(\alpha)=2a^2+10a+18>2a^2+2a+36=\text{Tr}(\gamma)$, thus we can exclude $\varepsilon=1$. Regarding $\varepsilon\neq 1$, the only possible candidates are again conjugates of $\rho^2\alpha''^2$ and $\rho^2\rho'^2\alpha''^2$. Nevertheless, by direct calculations, we can easily show that none of them is totally smaller than $\gamma$.    

We will proceed with $\alpha=-(a-1)\rho+\rho^2$. In this case, we have $\alpha^2>4a^2+8+\frac{4}{a^2}>\gamma$, thus we can exclude $\varepsilon=1$. For $\varepsilon\neq 1$, similarly, as before, we are left with $\rho^2\alpha''^2$ and $\rho^2\rho'^2\alpha''^2$. In this case, we indeed get an element, which is totally smaller than $\gamma$, and it is equal to
\[
\rho^2\rho'^2\alpha''^2=a^2+a-1+(a^2-a-3)\rho-(a-2)\rho^2.
\]

Let now $\alpha=-2\rho+\rho^2$. We can conclude that only the conjugates of $\rho^2\rho'^2\alpha''^2$ can be totally smaller than $\gamma$, from which only the one satisfies this condition, namely
\[
\rho''^2\rho^2\alpha'^2=a^2-a+(a^2-3a+1)\rho-(a-3)\rho^2.
\]

Therefore, it remains to consider the element $\alpha=-\rho+\rho^2$. Using Lemma \ref{lemma:a^4/a^2+a^2} and estimates (\ref{eq:estvalues}), we can easily derive that some conjugate of our element has to be of the form $\varepsilon\alpha''^2$ where $\varepsilon\in\{\rho^2\rho'^2,\rho^4\rho'^2,\rho^4\rho'^4\}$. From these nine cases, only two are actually totally smaller than $\gamma$, specifically
\begin{align*}
\rho'^2\rho''^2\alpha^2&=1-2\rho+\rho^2,\\
\rho''^2\rho^2\alpha'^2&=a^2+a+1+(a^2-a+1)\rho-(a-1)\rho^2.
\end{align*}
The others can be excluded by direct calculations, completing the proof.
\end{proof}

\subsection{Sums of $\sigma$-indecomposable integers} \label{subsec:sums}
In the previous subsections, we have found all the squares of $\sigma$-indecomposable integers $\beta$ for all signatures $\sigma$ for which we have $\gamma\succeq \beta^2$. Now we will consider possible $\sigma$-decomposable integers, which we can create from these $\sigma$-indecomposables and which are (possibly) totally smaller than $\gamma$. However, to do that, we must know the signatures of these elements, which can be found in Table \ref{tab:signs}.  

\begin{table}
\begin{tabular}{|c|c|c|}
\hline
$\sigma$-indecomposable integer $\beta$ & Signature of $\beta$ & Signature of $-\beta$\\
\hline
$1$ & $(+,+,+)$ & $(-,-,-)$\\
\hline
$\rho$ & $(+,-,-)$ & $(-,+,+)$\\
\hline
$\rho'\rho''(-\rho+\rho^2)$& $(+,-,-)$ & $(-,+,+)$\\
\hline
$\rho''\rho(-\rho'+\rho'^2)$  & $(-,-,+)$ & $(+,+,-)$\\
\hline
$\rho''\rho(-2\rho'+\rho'^2)$  & $(-,-,+)$& $(+,+,-)$\\
\hline
$\rho\rho'(-(a-1)\rho''+\rho''^2)$ & $(-,+,-)$ & $(+,-,+)$\\
\hline 
\end{tabular}

\caption{Signatures of $\sigma$-indecomposable integers whose squares are totally smaller than $\gamma$.} \label{tab:signs}
\end{table} 

\begin{lemma} \label{lemma:decompo}
Let $a\geq 15$ and $\gamma\succeq \omega^2$. If $\omega$ is $\sigma$-decomposable for some $\sigma$, then $\omega^2\in\{4,9\}$.
\end{lemma}

\begin{proof}
Now we will consider possible sums of our $\sigma$-indecomposable integers. Note that we can sum up only the elements with the same signatures. Moreover, the opposite signatures (i.e., which have all the signs opposite) behave in the same manner and give the same squares, and thus it suffices to study only one of each such a pair. 
\begin{enumerate}
\item Signature $(+,+,+)$ (respectively, $(-,-,-)$): Here we have only one element, namely $1$, which produces two $\sigma$-decomposable integers $2$ and $3$ whose squares $4$ and $9$ are totally smaller than $\gamma$.
\item Signature $(+,-,-)$ (respectively, $(-,+,+)$): The set of $\sigma$-indecomposables for this signature consists of the elements $\rho$ and $\rho'\rho''(-\rho+\rho^2)$. However, we can easily compute that
\begin{enumerate}
\item $(2\rho)^2>4(a+1)^2>\gamma$,
\item $(\rho+\rho'\rho''(-\rho+\rho^2))^2=1-4\rho+4\rho^2>4a^2+4a-3>\gamma$,
\item $(2\rho'\rho''(-\rho+\rho^2))^2=4-8\rho+4\rho^2>4a^2-8>\gamma$ 
\end{enumerate}
for $a\geq 15$.
Moreover, these results imply that our $\omega$ cannot also be a sum of more than two $\sigma$-indecomposable integers.
Thus, in this signature, none square of $\sigma$-decomposable integer is totally smaller than $\gamma$.
\item Signature $(-,-,+)$ (respectively, $(+,+,-)$): In this case, we consider exactly two $\sigma$-indecomposable integers $\rho''\rho(-\rho'+\rho'^2)$ and $\rho''\rho(-2\rho'+\rho')$. However, we can easily show that
\begin{enumerate}
\item $((2\rho''\rho(-\rho'+\rho'^2))^2)''>\gamma''$,
\item $(\rho''\rho(-\rho'+\rho'^2)+\rho''\rho(-2\rho'+\rho'^2))^2)''>\gamma''$,
\item $((2\rho''\rho(-2\rho'+\rho'^2))^2)''>\gamma''$ 
\end{enumerate}
for $a\geq 15$.
Thus we do not obtain any additional element.
\item Signature $(-,+,-)$ (respectively, $(+,-,+)$): This case contains exactly one $\sigma$-indecomposable integer, namely $\rho\rho'(-(a-1)\rho''+\rho''^2)$. However, it can be easily computed that 
\[
((2\rho\rho'(-(a-1)\rho''+\rho''^2))^2)''>\gamma''
\]
for $a\geq 15$, thus this case does not produce more elements to consider.     
\end{enumerate}

\end{proof}

\subsection{Proof of Theorem \ref{thm:main}}   

Using the results of Lemmas \ref{lemma:units}, \ref{lemma:inde} and \ref{lemma:decompo}, we can now prove Theorem $\ref{thm:main}$ stated in the introduction.

\begin{proof}[Proof of Theorem \ref{thm:main}]
In Subsections \ref{subsec:units}, \ref{subsec:inde} and \ref{subsec:sums}, we have found all the elements $\omega$ such that $\gamma\succeq \omega^2$ for $a\geq 15$. We have obtained the following squares:
\begin{enumerate}
\item rational integers $1$, $4$ and $9$,
\item squares of $\sigma$-indecomposable integers of the form 
\begin{enumerate}
\item $\rho^2$,
\item $1-2\rho+\rho^2$, 
\item $a^2+a+1+(a^2-a+1)\rho-(a-1)\rho^2$, 
\item $a^2-a+(a^2-3a+1)\rho-(a-3)\rho^2$,
\item $a^2+a-1+(a^2-a-3)\rho-(a-2)\rho^2$.
\end{enumerate} 
\end{enumerate} 
Using a computer program (all the calculations were performed in Mathematica), we can show that we get the same elements (and none more) also for $5\leq a\leq 14$. For $a=3$, we get two additional elements $20+11\rho-3\rho^2$ and $1+2\rho+\rho^2$, and for $a=4$, we obtain $1+2\rho+\rho^2$. Nevertheless, using a similar procedure as below, we can prove that even in these cases, we need at least $6$ squares to express $\gamma$. Thus, in the following, we will suppose $a\geq 5$.

Recall that $\gamma=a^2+a+8+(a^2-a+1)\rho+(2-a)\rho^2$.
The coefficient before $\rho$ of $\gamma$ is clearly odd, thus in every square decomposition of $\gamma$, we need at least one element with this coefficient odd. Looking at the list, this is satisfied by the elements 
\begin{align*}
&a^2+a+1+(a^2-a+1)\rho-(a-1)\rho^2,\\ 
&a^2-a+(a^2-3a+1)\rho-(a-3)\rho^2,\\
&a^2+a-1+(a^2-a-3)\rho-(a-2)\rho^2.
\end{align*}
Note that these elements are also the only ones, which have positive coefficient before $\rho$. However, for $a^2-a+(a^2-3a+1)\rho-(a-3)\rho^2$ and $a^2+a-1+(a^2-a-3)\rho-(a-2)\rho^2$, the value of this coefficients is strictly smaller than $a^2-a+1$. Thus, if our square decomposition of $\gamma$ contained one of these two element, some other summand would have to be one of these three above-mentioned elements.    Nevertheless, in that case, the coefficient before $1$ (we mean the coefficients in the basis $1,\rho,$ and $\rho^2$) is at least $2a^2-2a>a^2+a+8$ for $a\geq 5$. This is not possible since all the squares totally smaller than $\gamma$ have a non-negative coefficient before $1$. Hence no square decomposition of $\gamma$ can contain the elements $a^2-a+(a^2-3a+1)\rho-(a-3)\rho^2$ and $a^2+a-1+(a^2-a-3)\rho-(a-2)\rho^2$. 

It implies that one summand of our decomposition must be $a^2+a+1+(a^2-a+1)\rho-(a-1)\rho^2$, and we get 
\[
\gamma=a^2+a+1+(a^2-a+1)\rho-(a-1)\rho^2+\delta,
\] 
where $\delta=7+\rho^2$. Obviously, every square decomposition of $\delta$ may consist  of only the elements $1$, $4$, $\rho^2$ and $1-2\rho+\rho^2$ since the coefficient before $1$ of the other elements from the list is too large. Nevertheless, $1-2\rho+\rho^2$ cannot appear in this decomposition since its coefficient before $\rho$ is negative, and the remaining three integers have this coefficient equal to zero. 
Thus, only the elements $1$, $4$, and $\rho^2$ can appear in a square decomposition of $\delta$, and for that, we need at least $5$ of these elements. It implies that every square decomposition of $\gamma$ consists of at least $6$ non-zero squares, which together with the upper bound, gives $\P(\Z[\rho])=6$.   
\end{proof}

\section{The case $-1\leq a\leq 2$} \label{Sec:asmall}

We will now focus on the remaining cases of $a$, i.e., $-1\leq a\leq 2$. However, the situation is different here. Indeed, at least the element $\gamma$ can be expressed as a sum of less than $6$ squares for all of these cases, thus it cannot provide us the same lower bound as before. Moreover, based on computer experiments, we may propose that the Pythagoras number of $\Z[\rho]$ is even less than $6$. Nevertheless, our computer program searches for elements of small traces, and thus we cannot exclude that there exists an element that can be written as a sum of more squares and has a large trace. 

The lower bounds on $\P(\Z[\rho])$ for $-1\leq a\leq 2$ are provided in Table \ref{tab:lbsmalla}. We also show here an example of an element for which this lower bound is attained.

\begin{table}[h]
\begin{tabular}{|c|c|c|}
\hline
$a$ & $\P(\Z[\rho])$ & Example of element\\
\hline
$-1$ & $\geq 4$ & $7$\\
\hline
$0$ & $\geq 5$ & $-8\rho+8\rho^2$\\
\hline
$1$ & $\geq 5$ & $4-3\rho+2\rho^2$\\
\hline
$2$ & $\geq 5$ & $7+\rho^2$\\
\hline
\end{tabular}
\caption{The lower bound on $\P(\Z[\rho])$ for $-1\leq a\leq 2$ and an example of element for which this lower bound is attained.} \label{tab:lbsmalla}
\end{table}  

\section*{Acknowledgements}
The author is greatly indebted to Pavlo Yatsyna and Vítězslav Kala for their advice during the preparation of this paper.

\end{document}